\pdfoutput=1
\documentclass{amsart}

\usepackage[mathscr]{eucal}
\usepackage{amssymb}
\usepackage[dvipsnames]{xcolor} 
\usepackage[normalem]{ulem}
\usepackage{wasysym}
\usepackage{amsthm}
\usepackage{bbold}
\usepackage{comment}
\usepackage{enumitem}
\usepackage{amsmath}
\usepackage{tikz-cd}
\usepackage{stmaryrd} 
\usepackage{etoolbox} 
\usepackage{microtype}
\usepackage{adjustbox}
\usepackage{ mathdots }
\usepackage{bbm}
\usepackage{mathtools}

\usepackage[unicode]{hyperref} 

\definecolor{dark-red}{rgb}{0.5,0.15,0.15}
\definecolor{dark-blue}{rgb}{0.15,0.15,0.6}
\definecolor{dark-green}{rgb}{0.15,0.6,0.15}

\hypersetup{
    colorlinks, linkcolor=OliveGreen,
    citecolor=OliveGreen, urlcolor=OliveGreen
}

\usepackage[nameinlink,capitalise,noabbrev]{cleveref}

\usepackage[all]{xy}
\xyoption{line}
\usepackage{graphicx}
\usepackage{caption}

\numberwithin{equation}{section}

\newtheorem{Thm}[equation]{Theorem}
\newtheorem*{Thm*}{Theorem}
\newtheorem*{MainThm*}{Main Theorem}
\newtheorem{Prop}[equation]{Proposition}
\newtheorem{Lem}[equation]{Lemma}

\newtheorem*{Cor*}{Corollary}

\newtheorem*{Que*}{Question}

\theoremstyle{remark}
\newtheorem{Def}[equation]{Definition}

\newtheorem{Nota}[equation]{Notation}
\newtheorem{Rem}[equation]{Remark}

\tikzset{
    labelrotatebelow/.style={anchor=north, rotate=90, inner sep=1.0mm}
}
\tikzset{
    labelrotateabove/.style={anchor=south, rotate=90, inner sep=1.0mm}
}
\SelectTips{cm}{10}

\newcommand{\nc}{\newcommand}
\nc{\dmo}{\DeclareMathOperator}
\renewcommand{\emptyset}{\varnothing}

\usepackage[disable]{todonotes}

\dmo{\End}{End}
\dmo{\Mor}{Mor}
\dmo{\Hom}{Hom}
\dmo{\id}{id}
\dmo{\im}{im}
\dmo{\Ker}{Ker}
\dmo{\opname}{op}

\nc{\bbL}{\mathbb{L}}
\nc{\bbA}{\mathbb{A}}
\nc{\bbE}{\mathbb{E}}
\nc{\bbN}{\mathbb{N}}
\nc{\bbQ}{\mathbb{Q}}
\nc{\bbZ}{\mathbb{Z}}
\nc{\bbF}{\mathbb{F}}
\nc{\bbS}{\mathbb{S}}
\nc{\cat}[1]{\mathscr{#1}}

\nc{\Mid}{\,\big|\,}
\nc{\SET}[2]{\big\{\,#1\Mid#2\,\big\}}
\dmo{\Ho}{Ho}

\renewcommand{\leq}{\leqslant}

\newcommand{\sfL}{\mathsf{L}}

\newcommand{\sfW}{\mathsf{W}}
\newcommand{\sfF}{\mathsf{F}}
\newcommand{\sfC}{\mathsf{C}}
\newcommand{\sfR}{\mathsf{R}}

\newcommand{\Set}{\mathsf{Set}}
\newcommand{\Vect}{\mathsf{Vect}}

\newcommand{\bij}{\mathsf{bij}}
\newcommand{\inj}{\mathsf{inj}}
\newcommand{\surj}{\mathsf{surj}}
\newcommand{\any}{\mathsf{any}}
\newcommand{\types}{\mathsf{types}}
\newcommand{\pointed}{\mathsf{pointed}}
\newcommand{\iso}{\mathsf{iso}}
\newcommand{\mon}{\mathsf{mon}}
\newcommand{\epi}{\mathsf{epi}}
\renewcommand{\hom}{\mathsf{hom}}

\newcounter{enum-resume-hack}
\Crefname{Thm}{Theorem}{Theorems}
\Crefname{Prop}{Proposition}{Propositions}
\Crefname{Lem}{Lemma}{Lemmas}
\Crefname{thmx}{Theorem}{Theorems}

\begin{document}

\title{The nine model category structures on the category of sets}

\author{Omar Antol\'{\i}n-Camarena}
\address{Instituto de Matem{\'a}ticas, National Autonomous University of Mexico,
Mexico City, Mexico}
\email{omar@matem.unam.mx}
\urladdr{https://www.matem.unam.mx/~omar/}

\author{Tobias Barthel}
\address{Max Planck Institute for Mathematics, Bonn, Germany}
\email{tbarthel@mpim-bonn.mpg.de}
\urladdr{https://sites.google.com/view/tobiasbarthel/}

\date{\today}



\begin{abstract}
    We give a proof of the folklore theorem, attributed to Goodwillie, that there are precisely nine model structures on the category $\Set$ of sets. This result is deduced from a complete study of lifting problems and the ensuing classification of all weak factorization systems on $\Set$. Moreover, we determine the Quillen equivalences between these model structures and exhibit an explicit example of equivalent model structures that cannot be realized by a single Quillen adjunction. 
\end{abstract}

\maketitle

\tableofcontents

\section{Introduction}\label{sec:introduction}

Quillen \cite{Quillen1967} introduced model structures in order to provide an axiomatic framework for formalizing homotopical reasoning within a given category. This theory simultaneously generalizes homological algebra for categories of chain complexes and homotopy theory of topological spaces to a variety of algebraic, geometric, or topological contexts. In usual applications, one constructs a specific model structure for the purpose of capturing the homotopical features of the situation one wishes the study; however, we may wonder if and when it is possible to classify all possible model structures on a given category.

The prototypical example one might attempt such a classification for is the category $\Set$ of sets and functions between them. Perhaps surprisingly, in this case, one can completely determine all model structures:

\begin{Thm*}\label{thm:main}
    There are precisely six weak factorization systems and nine model structures on the category of sets, tabled in \cref{prop:wfs} and \cref{thm:modelstructures}.
\end{Thm*}

We do not claim any originality to this result; as far as we are aware, it was first proven by Tom Goodwillie \cite{GoodwillieMO2010} and it should be attributed to him, even though others might have known about it independently.

Since the classification of model structures on $\Set$ seems to hold little practical value in the way of applications---after all, one rarely does homotopy theory with sets---the interest shifts to the approach instead. Indeed, while enumerating weak factorization systems and model structures on a given category is in general out of reach, techniques for understanding them systematically may serve as a blueprint for similar classification results in other contexts. This is part of the emerging subject of \emph{homotopical combinatorics} \cite{BHOOR2024}, which among other things studies model structures on certain finite posets, see for instance \cite{FOOQW2022,BOOR2023}.

The classification naturally proceeds along the following three steps:
    \begin{enumerate}
        \item[(1)] Determine which pairs of morphisms $(\pi,\rho)$ satisfy the lifting property, meaning that every solid commutative square
        \[
            \xymatrix{
                A \ar[r]^g \ar[d]_\pi & X \ar[d]^\rho\\
                B \ar[r]^f \ar@{..>}[ru]|-{\exists ?\lambda} & Y,}
        \]
        admits a lift $\lambda$ making the two triangles commute.
        \item[(2)] Determine the poset of weak factorization systems $(\sfL,\sfR)$, which are pairs of classes of morphisms that satisfy mutual lifting properties and such that every morphism factors as a morphism of the left class $\sfL$ followed by a morphism from the right class $\sfR$.
        \item[(3)] Check which pairs of weak factorization systems are compatible in the sense that they give rise to a model structure. This relies on the equivalent description of a model structure, traditionally defined as consisting of a triple $(\sfC,\sfF,\sfW)$ verifying certain axioms, as instead being given by two weak factorization systems $(\sfC \cap \sfW, \sfF)$ and $(\sfC, \sfF \cap \sfW)$. 
    \end{enumerate}
For the category $\Set$ these steps are carried out in \cref{sec:lifting}, \cref{sec:wfs}, and \cref{sec:modelstructures}, respectively. Furthermore, we classify all Quillen equivalences between the resulting model structures in \cref{sec:quillenequivalenes}. As a byproduct, we obtain simple examples of:
    \begin{itemize}
        \item a pair of weak factorization systems of the form $(\sfC \cap \sfW, \sfF)$ and $(\sfC, \sfF \cap \sfW)$ for which $\sfW$ does \emph{not} satisfy 2 out of 3, and
        \item a pair of model categories connected by a zigzag of Quillen equivalences but with no direct Quillen equivalence between them.
\end{itemize}
This also highlights that a secondary aim of this note is pedagogical: Our arguments are self-contained and might be used by a novice as an entry point to the subject of model categories and homotopical combinatorics. We then conclude our paper in \cref{sec:misc} with a pointer to a couple of further classification results in this spirit.

\subsection*{Acknowledgements}

We would like to thank Scott Balchin for encouraging us to write up a proper account of this material and useful comments on an earlier draft. TB is grateful to the Max Planck Institute for Mathematics and is supported by the European Research Council (ERC) under Horizon Europe (grant No.~101042990). Part of the work on this paper was carried out at the Isaac Newton Institute for Mathematical Sciences, Cambridge,  during the programme `Equivariant homotopy theory in context', supported by EPSRC grant EP/Z000580/1.

\section{Lifting problems}\label{sec:lifting}

In this section, we determine which lifting problems in $\Set$ admit a solution.

\begin{Def}
    Let $\pi\colon A \to B$ and $\rho\colon X \to Y$ be two morphisms in a category $\cat C$. A \emph{lifting problem} is given by a solid commutative square in $\cat C$
        \begin{equation}\label{eq:liftingsquare}
            \vcenter{
            \xymatrix{
                A \ar[r]^g \ar[d]_\pi & X \ar[d]^\rho\\
                B \ar[r]^f \ar@{..>}[ru]|-{\exists ?\lambda} & Y,}
            }
        \end{equation}
    and asks if there is a \emph{lift} (or \emph{filler}) $\lambda\colon B \to X$ making the two resulting triangles commute, i.e., such that $\lambda \circ \pi= g$ and $\rho \circ \lambda = f$. We say that $\pi$ has the \emph{left lifting property} (LLP) against $\rho$ (resp.~$\rho$ has the \emph{right lifting property} (RLP) against $\pi$) if every lifting problem of the form \eqref{eq:liftingsquare} has a solution $\lambda$. If that is the case, we write $\pi \boxslash \rho$; note that this includes the case that the collection of lifting problems involving $\pi$ and $\rho$ is empty. Given a class of morphisms $\cat S$ in $\cat C$, we write 
        \[
            {\cat S}^{\boxslash} \coloneqq \SET{\rho}{\sigma \boxslash \rho \text{ for all } \sigma \in \cat S} = \bigcap_{\sigma \in \cat S}\{\sigma\}^{\boxslash};
        \]
    for the \emph{right orthogonal} of $\cat S$; the \emph{left orthogonal}  ${}^{\boxslash}\cat S$ is defined dually. Finally, we write $\overline{\cat S} \coloneqq {}^{\boxslash}(\cat S^{\boxslash})$.
\end{Def}

Working in the category $\cat C = \Set$ of sets for the remainder of this section, it is convenient to think of the data of a lifting problem \eqref{eq:liftingsquare} as a map $f\colon B\to Y$ together with a collection of maps on fibers
    \[
        (g_b\colon \pi^{-1}(b) \to \rho^{-1}(f(b)))_{b \in B}.
    \]
As a first step, we can characterize when such a lifting problem has a solution:

\begin{Lem}\label{lem:lifting}
    Let $\pi\colon A \to B$ and $\rho\colon X \to Y$ be two functions. A solution to the lifting problem \eqref{eq:liftingsquare} exists if and only if the following conditions hold:
        \begin{enumerate}
            \item for all $b \in B$, we have $|g(\pi^{-1}(b))| \leq 1$, i.e., $g$ collapses each non-empty fiber of $\pi$ to a single point in $X$; and
            \item for all $y \in \im(f)$, we have $\rho^{-1}(y) \neq \emptyset$, i.e., every point in the image of the bottom map has a non-empty fiber.
        \end{enumerate}
\end{Lem}
\begin{proof}
    If the two conditions are satisfied, first choose some element $x_{b} \in \rho^{-1}(f(b))$ for each $b \in B$. We can then construct a lift $\lambda \colon B \to Y$ via
        \[
            \lambda(b) \coloneqq   
                \begin{cases}
                    g(\pi^{-1}(b)) & \text{if } \pi^{-1}(b) \neq \emptyset; \\
                    x_{b} & \text{otherwise}.
                \end{cases}
        \]
    The verification that this indeed provides a solution to the lifting problem is straightforward. 

    Conversely, the existence of a lift implies the two conditions: the commutativity of the top triangle in \eqref{eq:liftingsquare} means that $g$ must collapse each non-empty fiber of $\pi$ to a single point in $X$. To be able to extend the solution to points of $B$ with empty fibers, these points must go to points of $Y$ possessing non-empty fibers.
\end{proof}

This lemma lets us describe precisely when two morphisms in $\Set$ have the lifting property against each other. 

\begin{Prop}\label{prop:lifting}
    Consider two functions $\pi\colon A \to B$ and $\rho\colon X \to Y$. Then $\pi \boxslash \rho$ if and only if one of the following conditions is satisfied:
        \begin{enumerate}
            \item $A \neq \emptyset = X$; or 
            \item out of $\pi$ and $\rho$, at least one is injective and at least one is surjective.   
        \end{enumerate}
\end{Prop}
\begin{proof}
    Consider the two conditions of \cref{lem:lifting}. We will determine when we can find a commuting square \eqref{eq:liftingsquare} with one of these two conditions failing. If both $\pi$ and $\rho$ have a fiber with at least two points, there is a lifting problem in which some non-empty fiber of $\pi$ is not collapsed to a single point; otherwise, this collapsing condition is satisfied for every lifting problem involving $\pi$ and $\rho$. In order to construct a lifting problem failing the second condition, we need both the following to be true:
        \begin{itemize}
            \item there exist $b_0 \in B$ and $y_0 \in Y$ such that $\pi^{-1}(b_0) = \emptyset$ and $\rho^{-1}(y_0) = \emptyset$; and 
            \item $A = \emptyset$ (so that all fibers of $\pi$ are empty) or $X \neq \emptyset$ (so that $X$ can accommodate the non-empty fibers of $\pi$).
        \end{itemize}
    Indeed, if $A = \emptyset$, we may take $g\colon B \to Y$ to map each $b \in B$ to $y_0$; note that $B$ is non-empty by the first condition. If $X \neq \emptyset$, choose  $x \in X$ and $y_1 = \rho(x)$. We then define  
        \[
            g(b) = 
                \begin{cases}
                    y_0 & \text{if } \pi^{-1}(b) = \emptyset; \\
                    y_1 & \text{otherwise},
                \end{cases}
        \]
    while $f(a) = x$ for each $a \in \pi^{-1}(b) \neq \emptyset$.

    Passing to the contrapositive, we see that $\pi \boxslash \rho$ if and only if the following two conditions (corresponding to the two conditions of the previous lemma) hold:
        \begin{enumerate}
            \item[(1)] at least one of $\pi$ or $\rho$ is injective; and
            \item[(2)] at least one of $\pi$ and $\rho$ is surjective, or
            $A \neq \emptyset = X$. 
        \end{enumerate}
    Since any function $\emptyset \to Y$ is injective, we can simplify this criterion to the statement of the proposition.
\end{proof}

\section{Weak factorization systems}\label{sec:wfs}

The goal of this section is to apply our analysis of lifting problems from the previous section to classify all weak factorization systems on the category of sets. For convenience, we recall their definition:

\begin{Def}\label{def:wfs}
    A \emph{weak factorization system} (WFS) on a category $\cat C$ consists of two classes of morphisms $\sfL$ and $\sfR$ of $\cat C$ satisfying the following two conditions:
        \begin{enumerate}
            \item Every morphism in $\cat C$ can be factored as $r\circ l$ with $r \in \sfR$ and $l \in \sfL$;
            \item $\sfL$ is the class of maps with the LLP against $\sfR$ and $\sfR$ is the class of maps with the RLP against $\sfL$; in symbols: $\sfL = {}^{\boxslash}\sfR$ and $\sfL^{\boxslash} = \sfR$.
        \end{enumerate}
    Finally, we say that a weak factorization system is \emph{(left) singly generated} if it is of the form $(\overline{\pi},\pi^{\boxslash})$ for some morphism $\pi$ in $\cat C$.
\end{Def}

A priori, not every morphism $\pi$ needs to generate a WFS, only those do for which every morphism can be written as a composite of a morphism in $\overline{\pi}$ and one in $\pi^{\boxslash}$; but it will turn out that in the category of sets all functions do generate a WFS.\footnote{In fact, the small object argument guarantees this for any locally presentable category.} We begin with the determination of all singly generated WFS on the category of sets. Once we are in the possession of the full list, it will turn out that there is just one more (i.e., not singly generated) WFS. We first introduce some notation.

\begin{Nota}\label{nota:wfs}
    We let $\bij$, $\inj$, $\surj$, and $\any$ denote the class of bijections, injections, surjections, or arbitrary morphisms in $\Set$. Moreover, we write $\inj_\emptyset$ for inclusions of the empty set into some other arbitrary set, while $\cat S_{\neq \emptyset}$ stands for the subset of morphisms in $\cat S$ with non-empty domain.
\end{Nota}

\begin{Lem}\label{lem:monogenwfs}
    Let $\pi\colon A \to B$ be a function. Then $\pi$ generates a weak factorization system on $\Set$ of the following form:
        \begin{enumerate}
            \item if $\pi$ is a bijection, then $(\overline{\pi},\pi^{\boxslash}) = (\bij,\any)$;
            \item if $\pi$ is an injection which is not surjective and
                \begin{enumerate}
                    \item $A = \emptyset$, then $(\overline{\pi},\pi^{\boxslash}) = (\inj,\surj)$;
                    \item $A \neq \emptyset$, then $(\overline{\pi},\pi^{\boxslash}) = (\inj_{\neq \emptyset}\cup \{\id_{\emptyset}\},\surj \cup \inj_{\emptyset})$;
                \end{enumerate}
            \item if $\pi$ is a surjection which is not injective, then $(\overline{\pi},\pi^{\boxslash}) = (\surj,\inj)$;
            \item if $\pi$ is not injective nor surjective, then $(\overline{\pi},\pi^{\boxslash}) = (\any_{\neq \emptyset}\cup \{\id_{\emptyset}\},\bij \cup \inj_{\emptyset})$.
        \end{enumerate}   
\end{Lem}
\begin{proof}
    In each of the these five cases, the description of $(\overline{\pi},\pi^{\boxslash})$ follows from \cref{prop:lifting}. For instance, suppose that $\pi$ is a bijection. This means that Condition $(b)$ of \cref{prop:lifting} is satisfied, hence any function has the RLP against $\pi$, i.e., $\pi^{\boxslash} = \any$. By considering some morphism with non-empty domain and which is neither injective nor surjective, another application of \cref{prop:lifting} then implies that $\overline{\pi} = {}^{\boxslash}\any = \bij$. If $\pi$ is injective but not surjective, we have to distinguish two cases: if $A = \emptyset$, then we are in the situation of \cref{prop:lifting}(b), so $\pi^{\boxslash} = \surj$ and consequently $\overline{\pi} = {}^{\boxslash}\surj = \inj$. If $A \neq \emptyset$, we have $\pi \boxslash \rho$ for some $\rho\colon X \to Y$ if and only if $X = \emptyset$ or $\rho$ is surjective. In other words, in this case we have $\pi^{\boxslash} = \surj \cup \inj_{\emptyset}$ and thus obtain 
        \[
            \overline{\pi} = {}^{\boxslash}(\surj \cup \inj_{\emptyset}) = {}^{\boxslash}(\surj) \cap {}^{\boxslash}(\inj_{\emptyset}) = \inj \cap (\any_{\neq \emptyset} \cup \surj) = \inj_{\neq \emptyset}\cup \{\id_{\emptyset}\}.
        \]
    The last two cases, (c) and (d), are dealt with similarly. Finally, it remains to observe that all of these pairs $(\bar{\pi}, \pi^{\boxslash})$ do in fact produce WFS, i.e., satisfy the factorization property \cref{def:wfs}(b) as well. 
\end{proof}

By direct inspection, we see that $(\any,\bij)$ is another, not singly generated WFS on $\Set$, and we claim that we have found all of them now:

\begin{Prop}\label{prop:wfs}
    There are precisely six weak factorization systems on the category of sets, given by the five singly generated ones from \cref{lem:monogenwfs} together with $(\any,\bij)$.
\end{Prop}
\begin{proof}
    We can order the six WFSs $(\sfL,\sfR)$ we have found above via inclusion of the left classes $\sfL$; i.e., we define $(\sfL,\sfR) \leq (\sfL',\sfR')$ for WFSs whenever $\sfL \subseteq \sfL'$ (or, equivalently, $\sfR' \subseteq \sfR$). The resulting poset then has the following Hasse diagram:
        \begin{equation}\label{eq:hasse}
            \vcenter{
            \xymatrixcolsep{1pc}
            
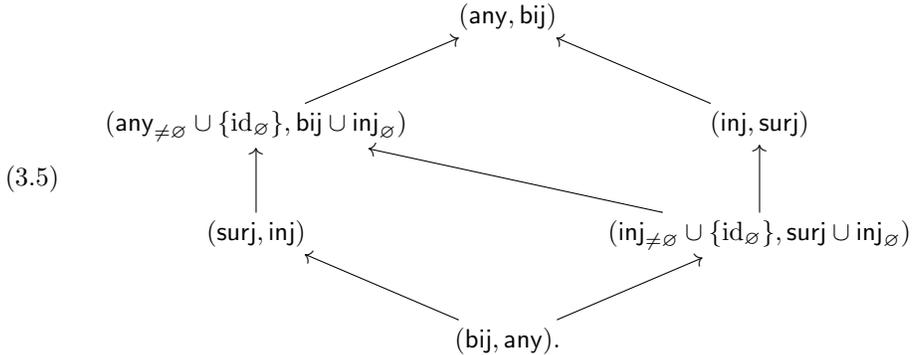
\xymatrix{
                & (\any,\bij) \\
                (\any_{\neq \emptyset}\cup \{\id_{\emptyset}\},\bij \cup \inj_{\emptyset}) \ar[ur] & & (\inj,\surj) \ar[ul] \\
                (\surj,\inj) \ar[u] & & (\inj_{\neq \emptyset}\cup \{\id_{\emptyset}\},\surj \cup \inj_{\emptyset}) \ar[u] \ar[ull] \\
                & (\bij,\any). \ar[ul] \ar[ur] &   \\             
            }
            }
        \end{equation}
    Now, given an arbitrary WFS $(\sfL,\sfR)$ on $\Set$, we have $\sfR = \bigcap_{\pi \in \sfL} \pi^\boxslash$, so that $\sfR$ is an intersection of the various right classes $\pi^\boxslash$ we found in \cref{lem:monogenwfs}. However, one can check by looking at the Hasse diagram that with the addition of the new minimal right class ${\bij}$, the collection of the right classes of weak factorization systems we found becomes closed under intersections. This shows that we have a complete list of WFSs for the category of sets, as desired.
\end{proof}

\section{Model structures}\label{sec:modelstructures}

With this preparation in hand, we are now ready to classify all model structures on the category of sets. We begin with the definition of model structures in terms of weak factorization systems:

\begin{Def}\label{def:modelstructure}
    Three classes of morphisms $\sfW$, $\sfC$ and $\sfF$ in a category $\cat C$ form a model structure $(\sfC,\sfF,\sfW)$ if and only if $\sfW$ satisfies the 2 out of 3 property and both $(\sfC \cap \sfW, \sfF)$ and $(\sfC, \sfF \cap \sfW)$ are WFSs. The elements of $\sfW$, $\sfC$, $\sfF$, $\sfC \cap \sfW$, and $\sfF \cap \sfW$ are referred to as weak equivalences, cofibrations, fibrations, acyclic cofibrations and acyclic fibrations, respectively.
\end{Def}

\begin{Rem}\label{rem:relation}
    For any model structure $(\sfC,\sfF,\sfW)$, we always have the useful relation $\sfW = (\sfF \cap \sfW) \circ (\sfC \cap \sfW)$. 
\end{Rem}

\begin{Thm}[Goodwillie]\label{thm:modelstructures}
    There are precisely nine model structures on the category of sets, as collected in the following table along with their homotopy categories:
        \begin{table}[h!]
            \centering
                \begin{tabular}{|lll |l|}
                    \hline
                    \text{Cofibrations} & \text{Fibrations} & \text{Weak equivalences} & \text{Homotopy category}\\
                    \hline
                    $\bij$ & $\any$ & $\any$ & $(-2)$-$\types$ \\
                    $\surj$ & $\inj$ & $\any$ & $(-2)$-$\types$\\
                    $\inj_{\neq\emptyset}\cup\{\id_\emptyset\}$ & $\surj\cup{\inj}_\emptyset$ & $\any$ & $(-2)$-$\types$\\
                    $\any_{\neq\emptyset}\cup{\{\id_\emptyset\}}$ & $\bij\cup\inj_\emptyset$ & $\any$ & $(-2)$-$\types$\\
                    $\inj$ & $\surj$ & $\any$ & $(-2)$-$\types$\\
                    $\any$ & $\bij$ & $\any$ & $(-2)$-$\types$\\
                    $\inj$ & $\surj\cup{\inj}_\emptyset$ & $\any_{\neq\emptyset}\cup{\{\id_\emptyset\}}$ & $(-1)$-$\types$\\
                    $\any$ & $\bij\cup\inj_\emptyset$ & $\any_{\neq\emptyset}\cup\{\id_\emptyset\}$ & $(-1)$-$\types$\\
                    $\any$ & $\any$ & $\bij$ & $0$-$\types$\\
                    \hline
                \end{tabular}
        \end{table}
\end{Thm}
\begin{proof}
    For any model structure $(\sfC,\sfF,\sfW)$, the constituent weak factorization systems satisfy $(\sfC \cap \sfW, \sfF) \leq (\sfC, \sfF \cap \sfW)$ (with the relation $\leq$ being defined as in the proof of \cref{prop:wfs}), so we can read off candidate pairs of WFSs from the Hasse diagram \eqref{eq:hasse}. We first consider the case that $(\sfC \cap \sfW, \sfF) = (\sfC, \sfF \cap \sfW)$. From \cref{rem:relation}, we then obtain
        \[
            \sfW = (\sfF \cap \sfW) \circ (\sfC \cap \sfW) = (\sfF \cap \sfW) \circ \sfC = \any,
        \]
    where the last equality is the factorization property of the WFS $(\sfC,\sfF \cap \sfW)$. Thus, there are six model structures where $\sfW = \any$ and $(\sfC,\sfF)$ is any of the six WFS found in \cref{prop:wfs}. Note that all of these have the terminal category $(-2)$-$\types$ as homotopy category.

    Now consider those model structures with $(\sfC \cap \sfW, \sfF) \lneq (\sfC, \sfF \cap \sfW)$. We distinguish the following cases:
        \begin{enumerate}
            \item Suppose $(\sfC \cap \sfW, \sfF) = (\bij, \any)$. In this case, we get $\sfW = (\sfF \cap \sfW) \circ (\sfC \cap \sfW) = \sfF \cap \sfW$, since every class contains all bijections and is closed under composition. So the right class of $(\sfC,\sfF \cap \sfW)$ must satisfy the 2 out of 3 property. However, among all the right classes we have, only $\any$, and $\bij$ satisfy 2 out of 3. We are in the case that the two WFSs of the model structure are different, so from the list of WFS we obtain precisely one new model structure in this case:
                \[
                    (\sfC, \sfF, \sfW) = (\any, \any, \bij).
                \]
            The corresponding homotopy category is the category of sets (or $0$-$\types$).
            \item Suppose $(\sfC, \sfF \cap \sfW) = (\any, \bij)$. Here, analogously to the previous case, we have $\sfW = \sfC \cap \sfW$. Of all the left classes we have, only $\bij$, $\any$ and $(\any_{\neq\emptyset}\cup \{\id_\emptyset\})$ satisfy 2 out of 3. So we get one new model structure:
                \[
                    (\sfC, \sfF, \sfW) = ({\any}, {\bij}\cup {{\inj}_\emptyset}, \any_{\neq\emptyset}\cup \{\id_\emptyset\}).
                \]
            The homotopy category of this model structure is the walking arrow, with one object for the empty set, one for all other sets, and whose only non-identity morphisms is the inclusion of the empty set into the non-empty set; this category is also known as $(-1)$-$\types$.
            \item Suppose that both constituent WFSs are not maximal and not minimal in the Hasse diagram \eqref{eq:hasse}. There are three pairs of such compatible WFSs:
                \begin{enumerate}
                    \item $(\sfC \cap \sfW, 
                    \sfF) = ({\surj}, {\inj})$ and $(\sfC, \sfF \cap \sfW) = (\any_{\neq\emptyset}\cup \{\id_\emptyset\}, {\bij}\cup \inj_\emptyset)$.
                    In this case, we get $\sfW = \surj\cup \inj_\emptyset$,  which does not satisfy 2 out of 3, so these WFSs cannot form a model structure. 
                    \item $(\sfC \cap \sfW, \sfF) = (\inj_{\neq\emptyset}\cup \{\id_\emptyset\}, \surj\cup \inj_\emptyset)$ and $(\sfC, \sfF \cap \sfW) = (\any_{\neq\emptyset}\cup \{\id_\emptyset\}, \bij\cup \inj_\emptyset)$. In this case, we get $\sfW = \inj_{\neq\emptyset}\cup \{\id_\emptyset\}$, which again does not satisfy 2 out of 3.
                    \item $(\sfC \cap \sfW, \sfF) = (\inj_{\neq\emptyset}\cup \{\id_\emptyset\}, \surj\cup \inj_\emptyset)$ and $(\sfC, \sfF \cap \sfW) = (\inj, \surj)$. We get $\sfW = \any_{\neq\emptyset}\cup \{\id_\emptyset\}$, which does satisfy 2 out of 3. Consequently, we obtain the last model structure:
                        \[
                            (\sfC, \sfF, \sfW) = (\inj, \surj\cup \inj_\emptyset, \any_{\neq\emptyset}\cup \{\id_\emptyset\}).
                        \]
                    As in case (b), the associated homotopy category is category of $(-1)$-$\types$. 
                \end{enumerate}
        \end{enumerate}
    This exhausts all possible cases, and thus concludes the proof.
\end{proof}

\begin{Rem}\label{rem:curiosity}
    The cases (c.1) and (c.2) in the proof above give examples of pairs of WFSs which are of the form $(\sfC \cap \sfW,\sfF)$ and $(\sfC, \sfF \cap \sfW)$ for three classes $\sfC$, $\sfF$, $\sfW$, and which satisfy $\sfW = (\sfC \cap \sfW) \circ (\sfF \cap \sfW)$, but which do not form a model structure because $\sfW$ does not satisfy 2 out of 3.
\end{Rem}

\section{Quillen equivalences}\label{sec:quillenequivalenes}

Morphisms and equivalences of model structures are given by Quillen adjunctions, whose definition we now recall:

\begin{Def}\label{def:quillenadjunction}
    Given two categories $\cat C$ and $\cat D$ equipped with model structures, an adjunction $F \dashv G$ between functors $F\colon \cat C \to \cat D$ and $G\colon \cat D \to \cat C$ is said to be a \emph{Quillen adjunction} if any of the following equivalent conditions is satisfied:
        \begin{enumerate}
            \item $F$ preserves cofibrations and acyclic cofibrations;
            \item $G$ preserves fibrations and acyclic fibrations;
            \item $F$ preserves cofibrations and $G$ preserves fibrations;
            \item $F$ preserves acyclic cofibrations and $G$ preserves acyclic fibrations.
        \end{enumerate}
    A Quillen adjunction $F \dashv G$ is called a \emph{Quillen equivalence} if either $F$ or $G$ induces an equivalence of homotopy categories (in which case both do). Finally, we say that two model structures are \emph{equivalent} if there exists a zig-zag of Quillen equivalences between them.  
\end{Def}

Finding all Quillen equivalences between the model structures on the category of sets is easy because there are not that many functors to consider: the left adjoint in any Quillen adjoint must preserve colimits and is thus determined by the value on the singleton set. Indeed, if that value is $A$, the functor must be $A \times -$, since any set $S$ is the coproduct of $|S|$ singletons.

\begin{Thm}\label{thm:quillenequivalences}
    Two model structures on $\Set$ are equivalent, i.e., related by a zig-zag of Quillen equivalences, if and only if they have the same class of weak equivalences. Any such equivalence is realized by a single Quillen equivalence except for in the following three cases, in which a zig-zag of length 2 is required:
        \begin{enumerate}
            \item $(\surj,\inj,\any)$ and $(\inj_{\neq \emptyset}\cup \{\id_{\emptyset}\},\surj \cup \inj_{\emptyset},\any)$;
            \item $(\surj,\inj,\any)$ and $(\inj,\surj,\any)$;
            \item $(\any_{\neq \emptyset}\cup \{\id_{\emptyset}\},\bij \cup \inj_{\emptyset},\any)$ and $(\inj,\surj,\any)$.
        \end{enumerate}
    Moreover, in all cases the Quillen adjunctions may be taken to be the identity functors $\id \dashv \id$. 
\end{Thm}
\begin{proof}
    By definition, equivalent model structures have equivalent homotopy categories, so we can read of from our classification \cref{thm:modelstructures} that two model structures on $\Set$ are equivalent only if they have the same class of weak equivalences. For the converse, we distinguish three cases, according to the three classes of weak equivalences that can arise in the classification:
        \begin{enumerate} 
            \item[(1)] For $\sfW=\bij$, there is only a single model structure, so the claim holds trivially. 
            \item[(2)] For $\sfW=\any_{\neq\emptyset}\cup\{\id_\emptyset\}$, the identity functor sends $\inj$ to $\any$ and $\bij \cup \inj_{\emptyset}$ to $\surj \cup \inj_{\emptyset}$, so is both a left and a right Quillen functor. 
            \item[(3)] For $\sfW=\any$, we see that the classes of cofibrations and fibrations in each of the six model structures must be of the form $(\sfC,\sfF)$ for one of the WFSs of \cref{prop:wfs}. It follows that the identity functor induces a Quillen equivalence between such whenever they are comparible in the Hasse diagram \eqref{eq:hasse}. Since the graph underlying the latter is connected, all of these model structures are equivalent, as claimed. 
        \end{enumerate}
    It remains to consider the three incomparable pairs of WFSs in Case $(3)$, as listed above. The Hasse diagram \eqref{eq:hasse} implies that each of them is related by a zig-zag of length 2, so in order to finish the proof, we need to prove that no direct Quillen equivalence between them exists. As remarked above, the left adjoint in any Quillen equivalence on $\Set$ has to be of the form $A \times -$ for some non-empty set $A$. By inspection, such a functor does not preserve any of the cofibrations participating in the pairs of model structures in $(a)$, $(b)$, and $(c)$ above: For example, for Case $(c)$, the injection $\emptyset \to  \{\ast\}$ is not mapped into $(\any_{\neq \emptyset}\cup \{\id_{\emptyset}\})$; for the converse, we may consider any surjection that is not injective. The other cases are dealt with similarly, thereby finishing the proof. 
\end{proof}

\section{Vistas}\label{sec:misc}

The arguments above may serve as a blueprint for similar classification results in other contexts. We collect two such examples which are again due to Goodwillie (see the comments in \cite{GoodwillieMO2024}) and mention some further work by others.

\subsection*{Pointed sets}

Write $\Set_{\ast}$ for the category of pointed sets and pointed maps, i.e., maps $f\colon (A,a) \to (B,b)$ which preserve the designated point: $f(a) = b$. As before (\cref{nota:wfs}), we write $\bij, \inj, \surj,$ and $\any$ for the classes of bijective, injective, surjective, or arbitrary morphisms of pointed sets. In addition, we let $\inj_{\ast}$ (resp.~$\inj_{\neq \ast}$) be the class of pointed maps $f\colon (A,a) \to (B,b)$ such that $f^{-1}(b) = \{a\}$ (resp.~$f$ is injective on $f^{-1}(B\setminus \{b\})$). Similar arguments to the ones above lead to a classification of model structures on $\Set_{\ast}$. This is due to Goodwillie; a write-up including detailed proofs can be found in notes by Idrissi \cite{Idrissi}. 

\begin{Thm}[Goodwillie]\label{thm:pointedsets}
    There are precisely six weak factorization systems on the category of pointed sets, as given in the following Hasse diagram:
        \begin{equation}\label{eq:pointedsets_hasse}
            \vcenter{
            \xymatrixcolsep{1pc}
            
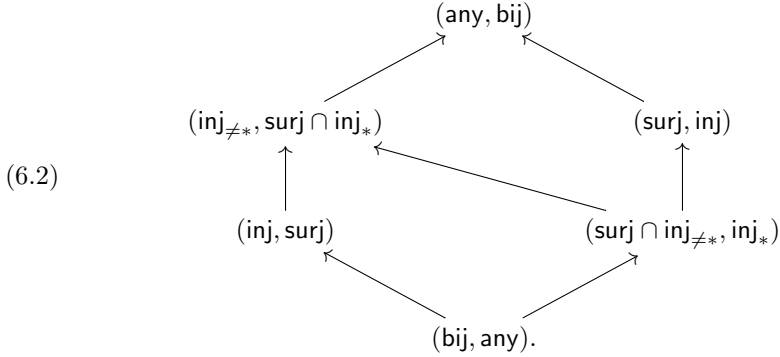
\xymatrix{
                & (\any,\bij) \\
                (\inj_{\neq \ast},\surj \cap \inj_{\ast}) \ar[ur] & & (\surj,\inj) \ar[ul] \\
                 (\inj,\surj) \ar[u] & & (\surj \cap \inj_{\neq \ast},\inj_{\ast}) \ar[u] \ar[ull] \\
                & (\bij,\any). \ar[ul] \ar[ur] &   \\             
            }
            }
        \end{equation}
    These weak factorization systems give rise to the following seven model structures, which provide a full list of model structures on the category $\Set_{\ast}$:
        \begin{table}[h!]
            \centering
                \begin{tabular}{|lll |l|}
                    \hline
                    \text{Cofibrations} & \text{Fibrations} & \text{Weak equivalences} & \text{Homotopy category}\\
                    \hline
                    $\any$ & $\any$ & $\bij$ & $\pointed$ $0$-$\types$\\
                    $\any$ & $\bij$ & $\any$ & $(-2)$-$\types$ \\
                    $\bij$ & $\any$ & $\any$ & $(-2)$-$\types$ \\
                    $\inj$ & $\surj$ & $\any$ & $(-2)$-$\types$ \\
                    $\surj$ & $\inj$ & $\any$ & $(-2)$-$\types$ \\
                    $\surj \cap \inj_{\neq \ast}$ & $\inj_{\ast}$ & $\any$ & $(-2)$-$\types$ \\
                    $\inj_{\neq \ast}$ & $\surj \cap \inj_{\ast}$ & $\any$ & $(-2)$-$\types$ \\
                    \hline
                \end{tabular}
        \end{table}
\end{Thm}
    
\begin{Rem}
    Variations on this theme are possible. For instance, one may consider the category of $G$-sets for $G$ a finite group; see \cite{MehrleSpitz2025} for results in this direction.
\end{Rem}

\subsection*{Vector spaces}\label{ssec:vectorspaces}

Let $k$ be a field and denote by $\Vect$ the category of $k$-vector spaces; since the next classification result is insensitive to the field, we will omit it from the notation. Write $\iso, \mon, \epi,$ and $\hom$ for the class of isomorphisms, monomorphisms, epimorphisms, or homomorphisms in $\Vect$, respectively. The classification of weak factorization systems and model structures on $\Vect$ is similar (but easier) than the one for $\Set$---this is again unpublished work of Goodwillie. 

\begin{Thm}[Goodwillie]\label{thm:vectorspaces}
    There are precisely four weak factorization systems on the category of vector spaces $\Vect$ (over any arbitrary field $k$), as given in the following Hasse diagram:
        \begin{equation}\label{eq:vectorspaces_hasse}
            \vcenter{
            \xymatrixcolsep{1pc}
            
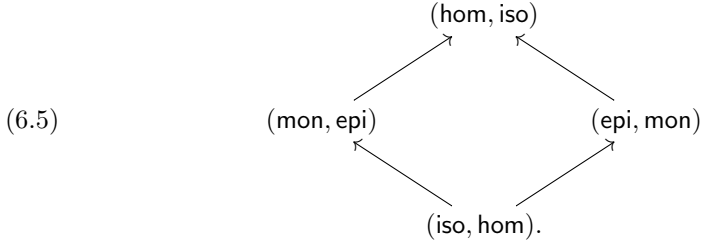
\xymatrix{
                & (\hom,\iso) \\
                (\mon,\epi) \ar[ur] & & (\epi,\mon) \ar[ul] \\
                & (\iso,\hom). \ar[ul] \ar[ur] &   \\             
            }
            }
        \end{equation}
    There are precisely five model structures on $\Vect$, as collected in the following table along with their homotopy categories: 
        \begin{table}[h!]
            \centering
                \begin{tabular}{|lll |l|}
                    \hline
                    \text{Cofibrations} & \text{Fibrations} & \text{Weak equivalences} & \text{Homotopy category}\\
                    \hline
                    $\hom$ & $\hom$ & $\iso$ & $\Vect$ \\
                    $\hom$ & $\iso$ & $\hom$ & $\mathsf{0}$ \\
                    $\iso$ & $\hom$ & $\hom$ & $\mathsf{0}$ \\
                    $\mon$ & $\epi$ & $\hom$ & $\mathsf{0}$ \\
                    $\epi$ & $\mon$ & $\hom$ & $\mathsf{0}$ \\
                    \hline
                \end{tabular}
        \end{table}
\end{Thm}

\begin{Rem}\label{rem:Artinian}
    Generalizing the above, we remark that Andrew Salch has announced a classification of all model structures on the category of $A$-modules, where $A$ is an Artinian commutative ring in which the square of each maximal ideal is zero. 
\end{Rem}

\bibliographystyle{alpha}
\bibliography{reference}

\end{document}